\documentclass[a4paper]{article}
\usepackage{amsmath, amsthm, amsfonts, amssymb, mathrsfs}
\usepackage[english]{babel}
\usepackage[all]{xy}
\usepackage{cite}
\usepackage{hyperref}
\newtheorem{teo}{Theorem}[section]
\newtheorem{prop}[teo]{Proposition}

\newtheorem{rem}[teo]{Remark}
\newtheorem{deff}[teo]{Definition}

\DeclareMathOperator{\vt}{\cdot_{\vartriangle}}
\DeclareMathOperator{\dt}{\cdot_{\emph{d}}}
\DeclareMathOperator{\gt}{\cdot_{\emph{D}}}

\newcommand{\CC}{\mathbb{C}}
\newcommand{\PP}{\mathbb{P}}
\newcommand{\NN}{\mathbb{N}}

\newcommand{\LL}{\mathcal{L}}

\newcommand{\dif}{\text{Diff}}
\newcommand{\der}{\text{Der}}
\newcommand{\lin}{\text{Lin}}

\begin{document}

\title{Deformations of the Exterior Algebra of Differential Forms}

\author{Ariel Molinuevo\thanks{The author was fully supported by CONICET, Argentina.}}
\date{}


\maketitle
\centerline{\emph{Departamento de Matem\'aticas, FCEyN, Universidad de Buenos Aires, Argentina.}}
\begin{abstract}
Let $D:\Omega\rightarrow{}\Omega$ be a differential operator defined in the exterior algebra $\Omega$ of differential forms over the polynomial ring $S$ in $n$ variables. In this work we give conditions for deforming the module structure of $\Omega$ over $S$ induced by the differential operator $D$, in order to make $D$ an $S$-linear morphism while leaving the $\CC$-vector space structure of $\Omega$ unchanged.
One can then apply the usual algebraic tools to study differential operators: finding generators of the kernel and image, computing a Hilbert polynomial of these modules, etc.

\

Taking differential operators arising from a distinguished family of derivations, we are able to classify which of them allow such deformations on $\Omega$. Finally we give examples of differential operators and the deformations that they induce.
\end{abstract}


\section{Introduction}

\

Let $S=\CC[x_1,\ldots,x_n]$ be the \emph{ring of polynomials} in $n$ variables and $\Omega=\bigoplus_{r\geq 0} \Omega^r$ the \emph{algebra of differential forms} of $S$ over $\CC$, where $\Omega^r$ denotes the module of $r$-differential forms. Since $\Omega^r$ has a natural structure of graded $S$-module, we will decompose it as $\Omega^r=\bigoplus_{\substack{b\geq 0}}\Omega^r(b)$, where we assign degree $+1$ to each $dx_i$. We stress the fact that this (second) grading is the one given by the \emph{Lie derivative} with respect to the \emph{radial vector field} $R=\sum_{i=1}^n x_i\frac{\partial}{\partial x_i}$, \emph{i.e.}, for $\tau\in\Omega^r(b)$ we have
\[
\LL_R(\tau) = b\ \tau
\]
as we will recall in eq. \ref{lie2}. In general, for a graded module, we will note in parenthesis the homogeneous component of the given degree; in the case of $\Omega$ it will always be the one given by the Lie derivative with respect to the radial vector field $R$.

\

For fixed $q,a$, suppose we are given a differential operator of order one $D:\Omega\rightarrow{}\Omega$, see Definition \ref{def-op-dif}, such that
\[
D(\Omega^r(b))\subset\Omega^{r+q}(b+a)
\]
for every $r,b\geq 0$.

\

Regarding $D$ as a morphism of $\CC$-vector spaces would probably lead to finding kernel and image of infinite dimension; on the other side, trying to compute these dimensions on each homogeneous component $\Omega^r(b)$ would force to know $D$ in complete detail. Instead, a module structure on $\Omega$ attached to $D$ would lead to more precise and computable information; such as a finite set of generators or a Hilbert polynomial associated to its kernel, or image, exposing discrete invariants.

\

Both sheaves of principal parts, see \cite{egaivIV}, and complexes of differential operators, see \cite{herrera-lieberman}, solve problems of linearization for $D$. In these works the solution is universal (in a categorical sense) and, as part of the construction, the operator $D$ changes as well as the domain. This fact makes the relation between $D$ and its linearized version difficult to track down.

\

Our approach to the problem of linearization is focused on the multiplicative action of $S$ over $\Omega$, leaving the operator $D$ invariant as well as the $\CC$-vector space structure on $\Omega$. 

\

As a first example, let us consider the usual \emph{exterior differential} $d:\Omega\rightarrow{}\Omega$ and denote by $i_R$ the \emph{contraction} with the radial vector field $R$. As we will show later in Section \ref{ext-diff}, for $f\in S(c)$ and $\tau\in\Omega^r(b)$, where the grading is taken with the Lie derivative with respect to the radial vector field R, in this work we introduce the following action
\begin{equation}\label{ej1}
f\dt \tau := \frac{b}{b+c}\ \left(\ f\tau\  +\ \frac{1}{b}\ df\wedge i_R\tau \ \right)
\end{equation}
which makes the exterior differential $d:\Omega\rightarrow{}\Omega$ a morphism of $S$-modules (when $b,c=0$ we adopt the usual multiplication). Even if we take a different approach, this can be readily verified by direct computations. 

\

In Section \ref{definition} we first give a formal definition of the deformations of the exterior algebra $\Omega$ over $S$ that we propose, Definition \ref{def-lin}, and prove a general condition condition for these modules to be finitely generated.

\

In Section \ref{theorem} we state our main result, Theorem \ref{teo1}, which is a classification of a distinguished class of differential operators that allow a linearized structure as in Definition \ref{def-lin}.

This classification exploits the decomposition of a differential operator in terms of a linear map plus a derivation which, in turn, can be decomposed as the Lie derivative plus a contraction with respect to vector valued differential forms. We give the details and definitions of this decomposition in Section \ref{definitions}

\

Finally, in Section \ref{applications} we present two examples of differential operators that allow the linearized structures that we defined.

\

\subsection{Geometric Motivation}

Even if our approach to the subject is purely algebraic, there is a geometric nature in our work that we would like to state here.

A codimension one foliation in projective space $\PP^{n-1}$ of degree $a-2$, $Fol^1(\PP^{n-1},$ $a-2)$, is given by a differential 1-form $\omega\in\Omega^1(a)$ such that it descends to projective space, \emph{i.e.}, $i_R\omega=0$, and such that verifies the Frobenious integrability condition  $\omega\wedge d\omega = 0$. As it is shown in \cite{fji}, the Zariski tangent space to the space of such foliations can be parameterized by
\[
 T_\omega Fol^1(\PP^{n-1},a-2) = \{ \eta\in\Omega^1(a)\ :\ i_R\eta=0 \text{ and } \omega\wedge d\eta+ d\omega\wedge \eta = 0\}.
\]
Using the second equation, C. Camacho and A. Lins-Neto, in  \cite{caln}, define the following notion of regularity of an integrable, homogeneous, differential 1-form and prove an associated stability result. 
By looking at $\omega$ as a homogeneous affine form in $\CC^{n}$, $\omega$ is said to be regular if for every $a<e$ the graded complex of homogeneous elements
\begin{equation}\label{reg-intro}
\xymatrix@R=0pt{
T(a-e) \ar[r]^-{} & \Omega^1(a) \ar[r]^-{} & \Omega^3(a+e)\\
X \ar@{|->}[r] & \LL_X(\omega) & &\\
& \save[] *{\eta\ }="s" \restore & \save[] +<40pt,0pt> *{\ \omega\vartriangle\eta:=\omega\wedge d\eta + d\omega\wedge \eta}="t" \restore & \ar@{|->}"s";"t"
}
\end{equation}
has trivial homology in degree 1, where we denote as $T:=(\Omega^1)^*$ to the module of \emph{vector fields} and assign degree $-1$ to each $\frac{\partial}{\partial x_i} := dx_i^*$.

\

Studying the function $a\longmapsto \varphi_{\omega}(a):=dim_\CC\left(Ker(\omega\vartriangle - )(a)\right)$, for every $a\in\NN$, it was clear the necessity of a better understanding of the differential operator $\omega\vartriangle - $ what led us to the present work (among other things, as we will briefly mention in \ref{reg-complex}).


\subsection*{Acknowledgements}
I would like to thank Matias del Hoyo, Alicia Dickenstein and Federico Quallbrunn for their suggestions and meaningful discussions.

\

\section{Basic Definitions and General Setting}\label{definitions}

\

Along this section we give the definitions of our main objects of study, which are derivations, differential operators and vector valued differential forms on $\Omega$. With these objects we state a decomposition theorem for a differential operator on $\Omega$, see Theorem \ref{decomp}. Such decomposition theorem, as far as we know, is missing in the literature.

\

Let us first recall some basic definitions, see \emph{e.g.} \cite{michor}:

\

A \emph{linear map of degree $q$} in $\Omega$ is a linear map $\varphi:\Omega\rightarrow{} \Omega$ such that $\varphi(\Omega^r)\subset \Omega^{r+q}$ for every $r\geq 0$.

\

A \emph{derivation of degree $q$} in $\Omega$ is a linear map $D_0:\Omega\rightarrow{}\Omega$ of degree $q$ such that
\[
D_0(\mu\wedge\tau) = D_0(\mu)\wedge\tau + (-1)^{rq}\  \mu\wedge D_0(\tau)
\]
for $\mu\in\Omega^r$ and $\tau\in\Omega^s$. If $D\left(\Omega^r(b)\right)\subset \Omega^{q+r}(b+a)$ we will say that $D$ has \emph{bidegree $(q,a)$} and denote by $\der^q(\Omega)$ and $\der^{(q,a)}(\Omega)$ the spaces of derivations of degree $q$ and bidegree $(q,a)$ respectively. In case that $D_0$ is $S$-linear we will say that $D_0$ is a \emph{linear derivation}.

\

We recall from the introduction that we will write as $T:=(\Omega^1)^*$ to the module of vector fields and assign degree $-1$ to each $\frac{\partial}{\partial x_i} := dx_i^*$.

In $\Omega$ we have the exterior differential $d$ which is a derivation of degree 1 and, for $X\in T$, the contraction $i_X$ which is a linear derivation of degree $-1$. 

\

Let $\varphi,\psi:\Omega\rightarrow{} \Omega$ be two linear maps of degree $q$ and $r$ respectively. We define the \emph{(graded) Lie bracket} $[\varphi,\psi]$ as
\[
[\varphi,\psi]:=\varphi\circ\psi - (-1)^{qr}\psi\circ\varphi.
\]

\

Following Cartan's formula, see \emph{e.g.} \cite{warner}, the \emph{Lie derivative} with respect to $X$, $\LL_X$, can be computed as
\begin{equation*}
\LL_X = [i_X,d] =  i_Xd + di_X
\end{equation*}
which is a derivation of degree $0$.

Taking $\tau\in\Omega^r(b)$, the Lie derivative with respect to the radial vector field $R$, verifies the well known formula, see \emph{e.g.} \cite{jou},
\begin{equation}\label{lie2}
\LL_R(\tau) = i_Rd\tau + di_R\tau = b \ \tau
\end{equation}
which decomposes in a unique way a differential form in a radial plus and exact form.

\

A linear derivation $D_0$ of degree $q$ is uniquely determined by its restriction to 1-forms $D_0:\Omega^1\rightarrow{} \Omega^{q+1}$, which can be viewed as an element $L\in\Omega^{q+1}\otimes T$. 

\

For $q\geq -1$, we will say that $L\in\Omega^{q+1}\otimes T$ is a \emph{vector valued differential form} and note $\Omega^{q+1}_{T}$ to the space of such elements. For $\tau\in\Omega^r$, we define the contraction $i_L\tau\in\Omega^{q+r}$ by the formula
\begin{align}\label{contraction}
i_L\tau &(X_1,\ldots,X_{q+r}) := \nonumber \\
&= \frac{1}{(q+1)!(r-1)!}\sum\limits_{\sigma\in S_{q+r}} sign(\sigma) \ \tau\left(L\left(X_{\sigma(1)},\ldots,X_{\sigma(q+1)}\right),X_{\sigma(q+2)},\ldots,X_{\sigma(q+r)}\right)
\end{align}
for $X_1,\ldots,X_{q+r}\in T$ and $S_{q+r}$ the permutation group of $q+r$ elements.

\

We recall from \cite{michor} the following two propositions that give a classification of derivations:


\begin{prop}\label{prop}If $L\in\Omega^{q+1}_{T}$, then $i_L\in\der^q (\Omega)$ and any linear derivation is of this form.
\end{prop}
\begin{proof}
We just notice that we are using the the identification of a differential $r$-form as an alternating map $\left(T\right)^{\otimes r}\rightarrow{} S$, see \emph{e.g.} \cite{warner}. For the rest, we follow \cite[Chapter IV, 16.2, p.~192]{michor}.
\end{proof}

\

Using eq. 3 and Proposition \ref{prop} we can define the Lie derivative for a vector valued differential form, $K\in\Omega^{q}_{T}$, as
\[
\LL_K:=[i_K,d]\in\der^q(\Omega).
\]

\

\begin{prop}\label{decomp1} If $D_0\in \der^q(\Omega)$, then there exists unique $K\in\Omega^{q}_{T}$ and $L\in\Omega^{q+1}_{T}$ such that
\begin{equation*}\label{der}
D_0 = \LL_K+i_L.
\end{equation*}
\end{prop}
\begin{proof}
See \cite[Chapter IV, 16.3, p.~193]{michor}.
\end{proof}

\

For $\tau\in\Omega$ we will denote by $\lambda_{\tau}$ the endomorphism of \emph{left multiplication} by $\tau $ in $\Omega$. It is immediate to see that a linear map $\varphi:\Omega\rightarrow{} \Omega$ of degree $q$ is $\Omega$-linear, \emph{i.e.} $\varphi(\mu) = \varphi(1)\wedge\mu$, if and only if
\[
[\varphi, \lambda_\tau] = 0
\]
for every $\tau\in \Omega$. We then define:

\

 A \emph{differential operator of order 1 and degree $q$}, or simply a \emph{differential operator of degree $q$}, in $\Omega$ is a linear map $D:\Omega\rightarrow{}\Omega$ of degree $q$ such that 
\begin{equation}
\label{def-op-dif}
[\ [D, \lambda_\mu]\ ,\lambda_\tau] = 0
\end{equation}
for all $\mu,\tau\in \Omega$. If $D\left(\Omega^r(b)\right)\subset\Omega^{q+r}(b+a)$ we will say that $D$ has bidegree $(q,a)$ and denote by $\dif^{\,q}(\Omega)$ and $\dif^{(q,a)}(\Omega)$ the spaces of differential operators of degree $q$ and bidegree $(q,a)$ respectively.

\

Even if the following proposition is very well known, we add a proof just to show that the sign rule arising from the skew commutativity of $\Omega$ does not make any conflicts.
\

\begin{prop}\label{decomp2} Let $D\in\dif^q(\Omega)$. Then $D$ can be decomposed as
\[
D = \left(D - \lambda_{D(1)}\right) + \lambda_{D(1)}
\]
where $D - \lambda_{D(1)}\in\der^q(\Omega)$ and $\lambda_{D(1)}$ is a linear map.
\end{prop}
\begin{proof}
Evaluating at 1 the formula $[\ [D, \lambda_\mu]\ ,\lambda_\tau] = 0 $ we get
\begin{align*}
D(\mu\wedge\tau) + D(1)\wedge\mu\wedge\tau = D(\mu)\wedge\tau + (-1)^{qr} \mu\wedge D(\tau).
\end{align*}
And by subtracting $-2\  D(1)\wedge\mu\wedge\tau$ in both sides we see that
\begin{align*}
D(\mu\wedge\tau) - D(1)\wedge\mu\wedge\tau = \big(D(\mu)-D(1)\wedge\mu\big)\wedge\tau + (-1)^{qr} \mu\wedge\big( D(\tau)-D(1)\wedge\tau\big)
\end{align*}
showing that $D-\lambda_{D(1)}\in\der^q(\Omega)$.
\end{proof}

\

As a direct corollary of Proposition \ref{decomp1} and Proposition \ref{decomp2}, we have the following decomposition for a differential operator:

\begin{teo}\label{decomp} Let $D\in\dif^q(\Omega)$. Then $D$ can be written as
\[
D = \LL_K + i_L + \lambda_{\mu}
\]
for unique $K\in\Omega^{q}_{T}$, $L\in\Omega^{q+1}_{T}$ and $\mu\in\Omega^q$.
\end{teo}

\

\section{Deformations of the Exterior Algebra}\label{definition}

\

In Definition \ref{def-lin} we give formal definition of the deformations of $\Omega$ induced by a differential operator $D\in\dif^{(q,a)}(\Omega)$ and then, in Proposition \ref{accion-gen}, we give conditions to these modules to be finitely generated.

As eq. \ref{ej1} shows, these deformations have some denominators in the formula, that can be zero in low degrees. Because of that, we first need a technical definition that will allow us to avoid this situation.

\

Two graded $S$-modules $M$ and $N$ are said to be \emph{stably isomorphic} if there exists an $n_0\in\NN$ such that $M(k)\simeq N(k)$ for every $k\geq n_0$.

\

Let $\widetilde{\Omega}$ be an algebra of differential forms stably isomorphic to $\Omega$. Without loss of generality we can assume that
\[
\widetilde{\Omega} = \bigoplus_{\substack{r\geq 0\\b\geq n_r}} \Omega^r(b)
\]
for some $n_r\in\NN$.

\

\begin{deff}\label{def-lin} Let $D\in\dif^{(q,a)}(\widetilde{\Omega})$. For $f\in S(c)$ and $\tau\in\Omega^r(b)$ we define the following action of $S$ in $\Omega^r(b)$
\[
f\gt \tau\ =\ \alpha \ f\tau\ +\ \beta\ df\wedge i_R\tau
\]
where $\alpha = \alpha(r,b,c)$ and $\beta = \beta(r,b,c)$ verify the conditions
\begin{enumerate}
\item[a)] $\alpha(-,-,0)= 1$
\item[b)] $D(f\gt\tau) \ = \ f\ \gt D(\tau)$
\item[c)] $(gf)\gt \tau = g\gt(f\gt \tau)$.
\end{enumerate}
We will note $\Omega_D$ to $\widetilde{\Omega}$ under this action from $S$.
\end{deff}

\

In case such $\alpha$ and $\beta$ exists, $\Omega_D$ gets a structure of a graded $S$-module extending the usual multiplication from $\CC$  and
\[
\xymatrix{\Omega_D \ar[r]^-{D} & \Omega_D}
\]
is $S$-linear.

\

It is clear that the modules $\Omega^r_D$ are stably free of rank $\binom{n}{r}$. Next we give a condition for $\Omega_D$ to be finitely generated.

\

\begin{prop}\label{accion-gen} Let $\Omega^r_D = \bigoplus\limits_{\substack{b\geq n_r}} \Omega^r_D(b)$ with $n_r>n$ and let $\alpha(r,b,c)$ and $\beta(r,b,c)$ as in Definition \ref{def-lin}. If for every $b\geq n_r$  we have
\[
\alpha(r,b,1)^2-\beta(r,b,1)^2 \neq 0
\]
then $\Omega^r_D$ is finitely generated by elements of degree $n_r$.
\end{prop}
\begin{proof}Let us fix ($r$ and) $b$ and denote $\alpha = \alpha(r,b,1)$ and $\beta = \beta(r,b,1)$. 
Consider the multi-index $\gamma = (\gamma_1,\ldots,\gamma_n)\in\NN_0$ such that $\sum\limits_{i=1}^n \gamma_i = g+1=b+1-r$ and let $I=\{i_1,\ldots,i_r\}\subset \{1,\ldots,n \}$.

\

Take $b\geq n_r$. Every element of $\Omega^r_D(b+1)$ it is a sum of elements of the form
\[
x^\gamma dx_I
\]
where we write $x^\gamma=\prod\limits_{i=1}^{n} x_i^{\gamma_i}$ and $dx_I = dx_{i_1}\wedge \ldots \wedge dx_{i_r}$.

\

Let $k\in\{1,\ldots,n\}$ such that $\gamma_k\neq 0$. Then we have the formula
\begin{align}\label{ec1}
x_k \gt \left(x^{\gamma-e_k} dx_I\right) &= \alpha\ x^\gamma dx_I + \beta\ \sum_{j=1}^r\ x^{\gamma-e_k+e_{i_j}}\ dx_k\wedge i_{\frac{\partial}{\partial x_{i_j}}}dx_I
\end{align}
where we note $e_i$ to the $i$-th canonical vector.

\

If $k\in I$, then eq. \ref{ec1} equals to $(\alpha + \beta)\ x^\gamma\ dx_I$. Since $\alpha^2-\beta^2\neq0$ we have
\begin{equation}\label{gen1}
\frac{1}{a+b}x_k \gt \left(x^{\gamma-e_k} dx_I\right) = x^\gamma\ dx_I 
\end{equation}
when $k\in I$ and $\gamma_k\neq 0 $.

\

On the other side, assume that for all $k$ such that $\gamma_k\neq 0$ we have $k\notin I$. By hypothesis we know $b\geq n_r>n$. Replacing $b>n$ in the equation $b=g+r$ we have $g>n-r$. Then, necessarily there exists $k$ such that $\gamma_k\geq 2$.

\

Let $\ell\in\{1,\ldots,r\}$  such that $\ell\in I$ and $k$ such that $\gamma_k\geq 2$. We have
\begin{align}\label{ec2}
x_\ell &\gt \left(x^{\gamma-e_k}dx_k\wedge i_{\frac{\partial}{\partial x_{\ell}}}dx_I\right) = \nonumber\\
&= \alpha\ x^{\gamma-e_k+e_\ell}\ dx_k\wedge i_{\frac{\partial}{\partial x_{\ell}}}dx_I + \beta\sum\limits_{t\in I\cup \{k\}} x^{\gamma-e_k+e_{t}}\ dx_\ell\wedge i_{\frac{\partial }{\partial x_{t}}}\left( dx_k\wedge i_{\frac{\partial}{\partial x_{\ell}}}dx_I\right) = \nonumber\\
&= \alpha\ x^{\gamma-e_k+e_\ell}\ dx_k\wedge i_{\frac{\partial}{\partial x_{\ell}}}dx_I + \beta\ x^\gamma\ dx_I - \beta\sum\limits_{j=1}^r x^{\gamma-e_k+e_{i_j}}\  dx_k\wedge i_{\frac{\partial}{\partial x_{i_j}}}dx_I.
\end{align}
Operating with eq. \ref{ec1} and eq. \ref{ec2} we get
\begin{align}\label{ec3}
\alpha\bigg[x_k \gt \left(x^{\gamma-e_k}\ dx_I\right)\bigg] -\beta \bigg[  x_\ell \gt \left(x^{\gamma-e_k}\ dx_k\wedge i_{\frac{\partial}{\partial x_{\ell}}}dx_I\right)\bigg] = \nonumber\\
= (\alpha^2-\beta^2)\ x^\gamma \ dx_I + (\alpha + \beta)\beta\ \sum\limits_{j=1}^rx^{\gamma-e_k+e_{i_j}}\ dx_k\wedge i_{\frac{\partial}{\partial x_{i_j}}}dx_I.
\end{align}

Let us call $J=(I\backslash\{\ell\})\cup\{k\}$. In each of the terms of the right of eq. \ref{ec3} we have $k\in J$ and $\gamma_k\geq 2$. Then $(\gamma-e_k+e_{i_j})_k\neq 0 $ and we can clear all the right terms of eq. \ref{ec3} following eq. \ref{gen1}, obtaining $(\alpha^2-\beta^2) x^\gamma dx_I$.

\

This way, every element is generated by elements of the previous degree starting from $b+1\geq n_r$ and the result follows.
\end{proof}

\begin{rem} We would like to point out that we do not know what happens with the converse statement, \emph{i.e.}, in the case where $\alpha(r,b,1)^2-\beta(r,b,1)^2=0$.
\end{rem}

\begin{rem} It is worth mentioning that $\Omega^r_D\neq \left(\Omega^1_D\right)^{\wedge r}$. For this, it is enough to compute the products
\[
(z\gt dx)\wedge dy  \qquad  dx\wedge (z\gt dy)
\]
and see that they are different.
\end{rem}

\

\section{Classification of Differential Operators}\label{theorem}

\

In Theorem \ref{teo1} we classify which differential operators allow a linearization as the one in Definition \ref{def-lin} for a distinguished class of differential operators.

\

There is a distinguished vector valued differential form $Id\in\Omega^1_{T}$, which is the one arising from the identity map on $\Omega^1\rightarrow{} \Omega^1$; then, $Id $ takes the form $Id = \sum_{i=1}^n dx_i\otimes\frac{\partial}{\partial x_i}$. Following eq. \ref{contraction}, we can apply $i_{Id}$ to an $r$-product of 1-differential forms $\tau = \tau_1\wedge\ldots\wedge\tau_r$ and get the formula
\[
i_{Id}(\tau) = r\ \tau.
\]
Then, it is immediate that the usual exterior differential can be computed as
\[
\LL_{Id}(\tau) = d\tau.
\]

\

The space of derivations arising from the module structure of $\Omega$ on the space generated by $(\LL_{Id},i_{Id})$ it is given by $(\LL_{\omega_1\wedge Id},i_{\omega_2\wedge Id})_{\{\omega_1,\omega_2\in\Omega\}}$; this can be easily seen by the following equalities, see \cite[Chapter IV, 16.7, Theorem, p.~194]{michor},
\[
\LL_{\omega\wedge Id} = \omega\wedge\LL_{Id} + (-1)^q d\omega\wedge i_{ Id} \qquad \text{and}\qquad i_{\omega\wedge Id} = \omega\wedge i_{Id}\ ,
\]
where $\omega\in\Omega^r(q)$. This allow us to define:

\begin{deff}We define the space of differential operators of bidegree $(q,a)$ associated to $Id\in\Omega^1_{T}$ as
\[
\dif^{(q,a)}_{Id}(\Omega) = \left\{ \widetilde{\omega}_1\wedge\LL_{Id} + \widetilde{\omega}_2\wedge i_{Id} + \lambda_{\widetilde{\mu}}:\text{for some }\widetilde{\omega}_1\in\Omega^{q-1}(a), \widetilde{\omega}_2,\widetilde{\mu}\in\Omega^q(a)\right\}
\]
\end{deff}

\

\begin{deff}For $q\geq 1$, we define the set $\lin^{(q,a)}(\Omega)$ of \emph{linearizable differential operators} of bidegree $(q,a)$
\begin{align*}\label{lin}
\lin^{(q,a)}(\Omega) =  \bigg\{&\omega_1\wedge\LL_{ Id} + \left(\frac{1}{q}d\omega_1+\omega_2\right)\wedge i_{ Id}  + (t\ \lambda_{d\omega_1}+\lambda_\mu): \text{ for some } \\
& \omega_1\in\Omega^{q-1}(a)\text{ and } \omega_2,\mu\in\Omega^{q}(a) \text{ such that } \\
&\hspace{4cm} i_R\omega_1=i_R\omega_2 = i_R\mu=0\text{ and }t\in\CC\bigg\}.
\end{align*}
\end{deff}

\

\begin{rem} Notice that $\lin^{(q,a)}\subsetneq\dif^{(q,a)}$. This can be seen by using the decomposition of eq. \ref{lie2} together with the conditions $i_R\omega_1=i_R\omega_2 = i_R\mu=0$, that fixes the exact form of the forms $\widetilde{\omega}_1$, $\widetilde{\omega}_2$ and $\widetilde{\mu}$ to $0$, $\frac{1}{q}d\omega_1$, and $td\omega_1$, respectively.

\end{rem}

\

\begin{rem} One can turn $\lin^{(q,a)}(\Omega)$ in a $\CC$-vector space in the following way: for $D$ and $D'$ in $\lin^{(q,a)}(\Omega)$ defined as
\begin{align*}
D &= \omega_1\wedge\LL_{ Id} + \left(\frac{1}{q}d\omega_1+\omega_2\right)\wedge i_{ Id} + (t\ \lambda_{d\omega_1}+\lambda_\mu)\\
D' &= \omega_1'\wedge\LL_{ Id} + \left(\frac{1}{q}d\omega_1'+\omega_2'\right)\wedge i_{ Id} + (t'\ \lambda_{d\omega_1'}+\lambda_{\mu'})
\end{align*}
we define the addition as
\[\begin{aligned}
&D+D' := (\omega_1+\omega_1')\wedge\LL_{ Id} + \left(\frac{1}{q}\ d\big(\omega_1+\omega_1'\big)+\big(\omega_2+\omega_2'\big)\right)\wedge i_{ Id} + \\ 
&\hspace{6cm} + \bigg[(t+t')\ \lambda_{d(\omega_1+\omega_1')}+ \lambda_{(\mu + \mu')}\bigg].
\end{aligned}
\]
Also, there is no ambiguity in the way these differential forms are written, since $\frac{1}{q}d\omega_1+\omega_2$ and $t\ d\omega_1+\mu$ are the addition of a radial plus an exact term, for which eq. \ref{lie2} assures uniqueness of writing.
\end{rem}

\

With the following theorem we classify which differential operators arising from $ Id$ can be linearized.

\begin{teo}\label{teo1} Let $D\in\dif_{Id}^{(q,a)}(\widetilde{\Omega})$. Then there exists $\alpha$ and $\beta$ that verify the conditions of Definition \ref{def-lin} making $D$ an $S$-linear operator if and only if $D\in\lin^{(q,a)}(\widetilde{\Omega})$, for $q\geq 1$. If $D\in\lin^{(q,a)}(\widetilde{\Omega})$ is given by
\[
D(\tau) = \omega_{1}\wedge \LL_{Id} +  \left(\frac{1}{q}\ d\omega_{1} + \omega_{2}\right)\wedge i_{Id} + (t \ \lambda_{d\omega_{1}} +\lambda_\mu)
\]
with $\omega_1\neq 0$, then $\alpha$ and $\beta$ can be chosen to be
\[
\alpha(r,b,c) := \frac{b-a\left(\frac{r}{q}+(-1)^{q}\,t\right)}{b+c-a\left(\frac{r}{q}+(-1)^{q}\,t\right)} \qquad \beta(r,b,c) := \frac{\alpha(r,b,c)}{b-a\left(\frac{r}{q}+(-1)^{q}\,t\right)}
\]
When $\omega_1=0$, the usual multiplication law can be used.
\end{teo}
\begin{proof} Take $D\in\dif_{Id}^{(q,a)}(\widetilde{\Omega})$. Then $D$ can be written as
\[
D = \omega_1\wedge \LL_{Id}+  \omega_2\wedge i_L + \lambda_\mu
\]
for some $\omega_1\in\Omega^{q-1}(a)$ and $\omega_2,\mu\in\Omega^q(a)$.

It will be convenient to decompose $\omega_1,\omega_2$ and $\mu$ in the following way
\begin{align*}
\omega_1 &= \omega_{1r} + \omega_{1d} \qquad \qquad \omega_2 = \omega_{2r} + \omega_{2d} + t_1 \ d\omega_{1r}\qquad \qquad \mu = \mu_{r} + \mu_{d} + t_2\ d\omega_{1r}
\end{align*}
where the subindex $r$ and $d$ denote radial and exact terms, and $\omega_{2d}$ and $\mu_{d}$ are linearly independent to $d\omega_{1r}$. For $\tau\in\Omega^r(b)$, we then have
\begin{align}
D(\tau) &= (\omega_{1r}+\omega_{1d})\wedge d\tau + \Big[ (rt_1+t_2) \ d\omega_{1r} + (r \omega_{2r}+\mu_{r}) + (r \omega_{2d}+\mu_{d})\Big]\wedge\tau \label{D-expr1}\\
D(\tau) &= (\omega_{1r}+\omega_{1d})\wedge d\tau + \Big[ (rt_1+t_2) \ d\omega_{1r} + \nu_r + \nu_d\Big]\wedge\tau \label{D-expr2}
\end{align}
where we are writing $\nu_r = r \omega_{2r}+\mu_{r}$ and $\nu_d = r \omega_{2d}+\mu_{d}$.

\

We now want to see under what conditions we have the equalities
\begin{enumerate}
\item[b)] $f\cdot D(\tau) = D(f\cdot\tau)$
\item[c)] $g\cdot(f\cdot \tau) = (gf)\cdot\tau$
\end{enumerate}
for the action defined in Definition \ref{def-lin}: $f\cdot \tau = \alpha\ f\ \tau + \beta\ df\wedge i_R\tau$.

\

From $b)$ we get 
\begin{align}\label{equ1}
&f\cdot D(\tau) = f\cdot\bigg\{(\omega_{1r}+\omega_{1d})\wedge d\tau + \Big[(rt_1+t_2)\ d\omega_{1r} + \nu_r + \nu_d \Big]\wedge\tau\bigg\} = \nonumber\\
&= \alpha(r+q,b+a,c)\ f\ \bigg\{(\omega_{1r}+\omega_{1d})\wedge d\tau + \Big[(rt_1+t_2)\ d\omega_{1r} + \nu_r + \nu_d \Big]\wedge\tau\bigg\} + \nonumber\\
&\hspace{.2cm}+ \beta(r+q,b+a,c)\ df\wedge i_R\Big((\omega_{1r}+\omega_{1d})\wedge d\tau + \Big[(rt_1+t_2)\ d\omega_{1r} + \nu_r + \nu_d \Big]\wedge\tau\Big) = \nonumber\\
&= \alpha(r+q,b+a,c)\ f\ \bigg\{(\omega_{1r}+\omega_{1d})\wedge d\tau + \Big[(rt_1+t_2)\ d\omega_{1r} + \nu_r + \nu_d \Big]\wedge\tau\bigg\} + \nonumber\\
&\hspace{.2cm}+\beta(r+q,b+a,c)df\wedge\bigg\{ i_R\omega_{1d}\wedge d\tau + (-1)^{q-1}(\omega_{1r}+\omega_{1d})\wedge i_Rd\tau +  \nonumber\\
&\hspace{.2cm} + \Big[a(rt_1+t_2)\ \omega_{1r}+i_R\nu_d\Big] \wedge \tau + (-1)^{q}\Big[(rt_1+t_2)\ d\omega_{1r} + \nu_r + \nu_d \Big]\wedge i_R\tau \bigg\}
\end{align}
and, writing $\alpha=\alpha(r,b,c)$,
\begin{align}\label{equ2}
&D(f\cdot\tau) = D\Big(\alpha\ f\ \tau + \beta\ df\wedge i_R\tau\Big) = \nonumber\\
&= \alpha \ (\omega_{1r}+\omega_{1d})\wedge df\wedge \tau + \alpha f \ (\omega_{1r}+\omega_{1d})\wedge d\tau - \beta \ (\omega_{1r}+\omega_{1d})\wedge df\wedge di_R\tau + \nonumber\\
&\hspace{1cm}+ \Big[(rt_1+t_2)\ d\omega_{1r} + \nu_r + \nu_d \Big]\wedge \Big(\alpha f\ \tau + \beta\ df\wedge i_R\tau\Big).
\end{align}

\

Taking the coefficients from every different term in eqs. \ref{equ1} and \ref{equ2}, we get the system
\[
\left\{
\begin{aligned}
&{\rm I})\ f\ \omega_{1r}\wedge d\tau: & \alpha(r+q,b+a,c)\ = \alpha(r,b,c)\ \\
&{\rm II})\ f\ \omega_{1d}\wedge d\tau: & \alpha(r+q,b+a,c)\ = \alpha(r,b,c)\ \\
&{\rm III})\ f\ d\omega_{1r}\wedge\tau: & \alpha(r+q,b+a,c)(rt_1+t_2)\ = \ \alpha(r,b,c)(rt_1+t_2)\\
&{\rm IV})\ f\ \nu_r\wedge\tau: & \alpha(r+q,b+a,c)\ = \alpha(r,b,c)\ \\
&{\rm V})\ f\ \nu_d\wedge\tau: & \alpha(r+q,b+a,c)\ = \alpha(r,b,c)\ \\
&{\rm VI})\ df\wedge i_R\omega_{1d}\wedge d\tau: &  \beta(r+q,b+a,c)\ = 0\ \\
&{\rm VII})\ df\wedge\omega_{1r}\wedge i_Rd\tau: & (-1)^{q-1}\beta(r+q,b+a,c)\ = (-1)^{q-1}\beta(r,b,c)\ \\
&{\rm VIII})\ df\wedge \omega_{1d}\wedge i_Rd\tau: & (-1)^{q-1}\beta(r+q,b+a,c)\ = (-1)^{q-1}\beta(r,b,c)\ \\
&{\rm IX})\ df\wedge\omega_{1r}\wedge\tau: & \beta(r+q,b+a,c)a(rt_1+t_2)\ =\hspace{1cm}\\
& & = (-1)^{q-1}(\alpha(r,b,c) -b\beta(r,b,c))\ \\
&{\rm X})\ df\wedge i_R\nu_d\wedge\tau: & \beta(r+q,b+a,c)\ = 0\ \\
&{\rm XI})\ df\wedge d\omega_{1r}\wedge i_R\tau: & \beta(r+q,b+a,c)(-1)^{q}(rt_1+t_2) = \hspace{1cm}\\
& & = \beta(r,b,c)(-1)^{q}(rt_1+t_2)\\
&{\rm XII})\ df\wedge \nu_r\wedge i_R\tau: & (-1)^{q}\beta(r+q,b+a,c) \ =\ (-1)^{q}\beta(r,b,c)\\
&{\rm XIII})\ df\wedge\nu_d\wedge i_R\tau: & (-1)^{q}\beta(r+q,b+a,c) \ =\ (-1)^{q}\beta(r,b,c)
\end{aligned}
\right.
\]
which can be reduced to
\[
\left\{
\begin{aligned}
&{\rm XI})\ df\wedge d\omega_{1r}\wedge i_R\tau: & (-1)^{q-1}\beta(r+q,b+a,c)\ = (-1)^{q-1}\beta(r,b,c)\ \\
&{\rm IV})\ f\ \nu_r\wedge\tau: & \alpha(r+q,b+a,c)\ = \alpha(r,b,c)\ \\
&{\rm IX})\ df\wedge\omega_{1r}\wedge\tau: & \beta(r+q,b+a,c)a(rt_1+t_2)\ = \hspace{1cm}\\
& & = (-1)^{q-1}(\alpha(r,b,c) -b\beta(r,b,c))\ \\
&{\rm X})\ df\wedge i_R\nu_d\wedge\tau: & \beta(r+q,b+a,c)\ = 0
\end{aligned}
\right.
\]

\

It is clear the sufficiency of these equalities to accomplish $b)$. For the necessity we proceed as follows:
\begin{itemize}
\item if $i_R\tau = 0$, then the terms involving eqs. ${\rm I}),\ldots,{\rm X})$ must coincide, leaving the terms of eqs. ${\rm XI}),{\rm XII})$ and ${\rm XIII})$ apart which can be considered in another system. Since these equations are linearly independent by hypothesis, we get that eq. ${\rm XI})$ must be satisfied.
\item if $\tau=df\wedge\rho_d$, for some exact differential form $\rho_d$, then only eqs. ${\rm III}),{\rm IV})$ and ${\rm V})$ survive. Also they are linearly independent by hypothesis, then eq. ${\rm IV})$ must be also satisfied.
\item by the previous arguments, we can clear off eqs. ${\rm I})$ to ${\rm V})$, ${\rm VII}),{\rm VIII})$ and ${\rm XI})$,${\rm XII})$,\linebreak${\rm XIII})$. Taking now $\tau=\tau_d$ we get only eqs. ${\rm IX})$ and ${\rm X})$ which are also linearly independent.
\end{itemize}

\

Using eq. ${\rm IX})$ we can clear $\beta$ as
\[
\beta(r,b,c) = \frac{\alpha(r,b,c)}{b-(-1)^qa(rt_1+t_2)}.
\]
From this equality and eqs. ${\rm I})$ and ${\rm VII})$ we get the formula
\[
\frac{\alpha(r,b,c)}{b-(-1)^qa(rt_1+t_2)} = \frac{\alpha(r,b,c)}{(b+a)-(-1)^qa\big((r+q)t_1+t_2)\big)}
\]
from where we are able to clear $t_1$ as $\frac{(-1)^q}{q}$ and obtain the system
\begin{equation}\label{sis-1}
\left\{
\begin{aligned}
& \alpha(r,b,c)\ = \alpha(r+q,b+a,c)\ \\
& \beta(r,b,c) \ = \ \frac{\alpha(r,b,c)}{ b-a\left(\frac{r}{q}+(-1)^{q}\,t_2\right)}\\
\end{aligned}
\right.
\end{equation}
From eqs. ${\rm VI})$ and ${\rm X})$ we see that $\omega_{1d}=\nu_d=0$, since $\alpha$ and $\beta$ must be non trivial. Then recalling the expressions of eqs. \ref{D-expr1} and \ref{D-expr2}, the differential operator $D$ must be of the form
\begin{align*}
D(\tau) &= \omega_{1r}\wedge d\tau + \Big[ \left(\frac{r}{q}+t_2\right) \ d\omega_{1r} + \nu_r \Big]\wedge\tau\\
D(\tau) &= \omega_{1r}\wedge d\tau +  r\left(\frac{1}{q}\ d\omega_{1r} + \omega_{2r}\right)\wedge\tau + (t_2 \ d\omega_{1r} +\mu_{r})\wedge\tau.
\end{align*}
This way, we have that
\[
D = \omega_{1r}\wedge \LL_{Id} +  \left(\frac{1}{q}\ d\omega_{1r} + \omega_{2r}\right)\wedge i_{Id}+ (t_2 \ \lambda_{d\omega_{1r}} +\lambda_{\mu_{r}})
\]
showing that $D\in\lin^{(q,a)}(\widetilde{\Omega})$.

\

From $c)$ we have
\begin{align}\label{equ3}
g\cdot&(f\cdot\tau) = g\cdot\big(\alpha(r,b,c)\ f\ \tau + \beta(r,b,c)\ df\wedge i_R\tau\big) = \nonumber\\
&=\alpha(r,b+c,e)g\ \Big(\alpha(r,b,c)\ f\ \tau + \beta(r,b,c)\ df\wedge i_R\tau\Big) + \nonumber\\
&\hspace{1cm}+ \beta(r,b+c,e)\ dg\wedge i_R\Big(\alpha(r,b,c)\ f\ \tau + \beta(r,b,c)\ df\wedge i_R\tau\Big) = \nonumber\\
&= \alpha(r,b+c,e)\alpha(r,b,c)\ gf\ \tau + \alpha(r,b+c,e)\beta(r,b,c)\ g\ df\wedge i_R\tau \ + \nonumber\\
&\hspace{0cm}+ \beta(r,b+c,e)\alpha(r,b,c)\ f \ dg\wedge i_R\tau + \beta(r,b+c,e)\beta(r,b,c)c\ f\ dg\wedge i_R\tau 
\end{align}
and
\begin{align}\label{equ4}
(gf)&\cdot\tau = \alpha(r,b,c+e)\ gf\ \tau + \beta(r,b,c+e)\ d(gf)\wedge i_R\tau = \nonumber\\
&= \alpha(r,b,c+e)\ gf\ \tau + \beta(r,b,c+e)g\ df\wedge i_R\tau + \beta(r,b,c+e)f\ dg\wedge i_R\tau.
\end{align}

\

Again, joining eqs. \ref{equ3} and \ref{equ4} as before we get the following  system of equations
\begin{equation}\label{sis-2}
\left\{
\begin{aligned}
&{\rm I})\ gf\ \tau: &\alpha(r,b+c,e)\alpha(r,b,c)\ = \ \alpha(r,b,c+e)\\
&{\rm II})\ g\ df\wedge i_R\tau: &\alpha(r,b+c,e)\beta(r,b,c)\ = \ \beta(r,b,c+e)\\
&{\rm III})\ f\ dg\wedge i_R\tau: &\Big(\beta(r,b+c,e)\alpha(r,b,c)+\beta(r,b+c,e)\beta(r,b,c)c\Big)\ = \\
&  & =\ \beta(r,b,c+e)
\end{aligned}
\right.
\end{equation}
which implies condition $c)$.

\

For the necessity we can assume $i_R\tau=0$ from where we get eq. ${\rm I})$. Removing that equation from the system, we can choose $f$ and $g$ linearly independent and we are done.

\

%

Putting together eqs. \ref{sis-1} and \ref{sis-2} we get the conditions
\[
\mathcal{S}:\ \left\{
\begin{aligned}
&\alpha(r,b,c)\ = \ \alpha(r+2,b+a,c) \ \\
& \alpha(r,b,c)\ = \ \frac{\alpha(r,b,c+e)}{\alpha(r,b+c,e)} \\
&\beta(r,b,c) \ = \ \frac{\alpha(r,b,c)}{ b-a\left(\frac{r}{q}+(-1)^{q}\,t_2\right)}\\
\end{aligned}
\right.
\]

\

The first two equations suggest a linear relation between the first two coordinates, for which the denominator of the third equation propose a formula for that. The second equation suggest a multiplicative relation between the last two coordinates.

\

As stated in the theorem, a formula that satisfies system $\mathcal{S}$ can be given by
\[
\alpha(r,b,c) := \frac{b-a\left(\frac{r}{q}+(-1)^{q}\,t_2\right)}{b+c-a\left(\frac{r}{q}+(-1)^{q}\,t_2\right)} \qquad 
\beta(r,b,c) := \frac{\alpha(r,b,c)}{b-a\left(\frac{r}{q}+(-1)^{q}\,t_2\right)}
\]
which clearly verifies condition $a)\alpha(-,-,0)=1$ of Definition \ref{def-lin}.

\

For the case $\omega_1=0$, we can choose $\alpha=1$ and $\beta=0$ which reduces to the usual multiplication.
\end{proof}

\begin{rem} A more general formula to the one given in the previous theorem can be given by, for an appropriate $t$, 
\[
\alpha(r,b,c) := \frac{F\Big(L\big(b\big)-L\left(a\left(\frac{r}{q}+(-1)^{q}\,t\right)\right)}{F\Big(L\big(b+c\big)-L\left(a\left(\frac{r}{q}+(-1)^{q}\,t\right)\right)} \quad
\beta(r,b,c) := \frac{\alpha(r,b,c)}{ b-a\left(\frac{r}{q}+(-1)^{q}\,t\right)}
\]
where $F$ is any function and $L$ is a linear function.
\end{rem}

\

\section{Applications}\label{applications}

\

Along this section we show two different applications of the formula for deformations that we gave in Definition \ref{def-lin}.
The first one shows how to linearize the usual exterior differential as we mentioned in eq. \ref{ej1}. The second one is related to the regularity complex of Camacho and Lins-Neto given in the introduction, see eq. \ref{reg-intro}.

\

\subsection{Exterior differential}\label{ext-diff}

\

\begin{deff}\label{def-d} For $f\in S(c)$ and $\tau\in\Omega^r(b)$, we define
\begin{equation*}
f\dt \tau := \frac{b}{b+c}\ \left(\ f\tau\  +\ \frac{1}{b}\ df\wedge i_R\tau \ \right)
\end{equation*}
when $b$ or $c$ is not null and define $f\cdot_d\tau = f\tau$ when $b,c=0$. We denote by $\Omega^r_d=\Omega^r$ and $\Omega_{d}=\bigoplus_{r\geq 0} \Omega^r_d$ to the $\CC$-vector spaces with this action from $S$.
\end{deff}

\

We then have:

\begin{prop}\label{acc-3} The exterior differential
\[
\xymatrix{\Omega_d \ar[r]^d & \Omega_d}
\]
is a morphism of $S$-modules and the $S$-modules $\Omega^r_d$ are finitely generated.
\end{prop}
\begin{proof}
It is clear that $d \in\lin^{(1,0)}(\Omega)$. Then, following Theorem \ref{teo1}, we get the formula proposed in Definition \ref{def-d}.

\

To see that $\Omega_d$ is finitely generated we have that  $\alpha^2(r,b,1)-\beta^2(r,b,1)=0$ if and only if
\[
\left(\frac{b-1}{b}\right)^2 - \left(\frac{b-1}{b^2}\right)^2 = \frac{(b-1)^2(b^2-1)}{b^4} = 0.
\]
\[
\left(\frac{b}{b+1}\right)^2 - \left(\frac{b}{(b+1)^2}\right)^2 = \frac{b^2[(b+1)^2-1]}{(b+1)^4} = 0.
\]

Then, the conditions of Proposition \ref{accion-gen} are verified for $b\geq 1$ and the result follows.
\end{proof}

\

\subsection{Regularity complex}\label{reg-complex}

\

In \cite{moli} we introduce a long complex $C^\bullet(\omega)$ of differential operators from the short complex of Camacho and Lins-Neto, see \cite[Section III.I, p.~17]{caln} or eq. \ref{reg-intro}. In the same work we give $\CC$-linear isomorphisms of $C^\bullet(\omega)$ to an $S$-linear complex and use it to obtain geometric information of the singular locus of the foliation defined by $\omega$ as well as its relation with first order unfoldings of $\omega$. Here we expose a different approach to linearize the complex $C^\bullet(\omega)$.

\

Let us recall from \cite[Section 6.1, p.~19]{moli} the following definition:

\

Let $\omega\in\Omega^1(a)$ such that $i_R\omega = 0$ and $\omega\wedge d\omega = 0$. We define the differential operator $\omega\vartriangle \ \in\dif^{(2,a)}(\Omega)$ as
\[
\xymatrix@R=0pt@C=50pt{ \Omega^r \ar[r]^{\omega\vartriangle} & \Omega^{r+2} \\
\save[] *{\tau\ }="s" \restore & \save[] +<50pt,0pt> *{\ \omega\vartriangle\tau:=\omega\wedge d\tau + \kappa(r)\ d\omega\wedge \tau}="t" \restore & \ar@{|->}"s";"t"
}
\]
where $\kappa(r) := \frac{r+1}{2}$.

\

Using \cite[Proposition 6.1.2, p.~19]{moli} we know that $\omega\vartriangle$ defines a differential of a complex of $\CC$-vector spaces which allow us to define the graded complex  $C^\bullet(\omega)$ as
\[
\xymatrix{
C^\bullet(\omega):& T\ar[r]^-{\omega\vartriangle} & \Omega^1 \ar[r]^-{\omega\vartriangle} & \Omega^3 \ar[r]^-{\omega\vartriangle} & \ldots & &
}
\]
where the 0-th differential is defined as $\omega\vartriangle X:=\LL_X(\omega)= i_Xd\omega + di_X\omega$.

\

\begin{deff}\label{def-t} For $f\in S(c)$ and $\tau\in\Omega^r(b)$ we define the action
\begin{equation*}
f\vt \tau := \frac{1}{b+c-\kappa(r)a}\ \Big[\  \big(b-\ \kappa(r)a\big) f\tau\  +\  df\wedge i_R\tau\ \Big]
\end{equation*}
and denote by $\Omega^r_{\omega\vartriangle}= \bigoplus_{b> \kappa(r)a }\Omega^r(b)$ and $\Omega_{\omega\vartriangle}=\bigoplus_{r\geq 0}\Omega^r_{\omega\vartriangle}$ to the $\mathbb{C}$-vector spaces with this action from $S$.

For $f\in S(c)$ and $X\in T(b)$ we also define
\[
f\vt X = \frac{b}{b+c} \ fX
\]
and denote by $T_{\omega\vartriangle}=\bigoplus\limits_{b\geq 1}T(b)$ to the $\mathbb{C}$-vector space with this action from $S$.
\end{deff}

\

We then have:

\begin{prop} The complex 
\[
\xymatrix@R=0pt@C=35pt{
C^\bullet_{\omega\vartriangle}(\omega):& T_{\omega\vartriangle} \ar[r]^-{\omega\vartriangle} & \Omega^1_{\omega\vartriangle} \ar[r]^-{\omega\vartriangle} & \Omega^3_{\omega\vartriangle} \ar[r]^-{\omega\vartriangle} & \ldots\\
}
\]
is a complex of $S$-modules and the $S$-modules $\Omega^r_{\omega\vartriangle}$  and $T_{\omega\vartriangle}$ are finitely generated.
\end{prop}
\begin{proof}
 Writing $\omega\vartriangle\ $as
\[
\omega\vartriangle\  = \omega\wedge\LL_{Id} + \frac{1}{2}d\omega\wedge i_{Id} + \frac{1}{2}\lambda_{d\omega}
\]
it is clear that $\omega\vartriangle\ \in\lin^{(2,a)}(\Omega)$. Then, following Theorem \ref{teo1}, we get the formula proposed in Definition \ref{def-t}.

\

It is also clear that $T_{\omega\vartriangle}$ is finitely generated. For $\Omega^r_{\omega\vartriangle}$ we have that $\alpha^2(r,b,1)-\beta^2(r,b,1)=0$ if and only if
\begin{align*}
\left(\frac{b-\kappa(r)a}{b+1-\kappa(r)a}\right)^2-\left(\frac{1}{b+1-\kappa(r)a}\right)^2 = \frac{\big(b-\kappa(r)a\big)^2-1}{(b+1-\kappa(r)a)^2} = 0. 
\end{align*}
Then, the conditions of Proposition \ref{accion-gen} are verified for $b>1+\kappa(r)a$ and the result follows.
\end{proof}



\begin{tabular}{l l}
Ariel Molinuevo$^*$  &\textsf{amoli@dm.uba.ar}\\
&\\
$^*$\textsc{Departamento de Matem\'atica} & \\
\textsc{Pabell\'on I} &\\
\textsc{Ciudad Universitaria} &\\
\textsc{CP C1428EGA} &\\
\textsc{Buenos Aires} &\\
\textsc{Argentina} &
\end{tabular}

\end{document}